\documentclass[11pt]{amsart}

\usepackage{amscd,amssymb,mathrsfs}
\usepackage{enumerate}
\usepackage{xcolor}
\usepackage{tikz-cd}
\usepackage{tikz}
\usetikzlibrary{matrix}

\usepackage[hidelinks]{hyperref}

\newcommand{\Z}{\mathbb Z}
\newcommand{\Q}{\mathbb Q}
\newcommand{\R}{\mathbb R}
\newcommand{\C}{\mathbb C}
\renewcommand{\P}{\mathbb P}

\newcommand{\bL}{\mathbb L}

\newcommand{\bV}{\mathbb V}

\newcommand\kk{{\Bbbk}}

\newcommand{\cM}{\mathcal M}

\newcommand{\sB}{\mathscr B}

\newcommand{\g}{\mathfrak{g}}

\newcommand{\p}{\mathfrak{p}}

\newcommand{\e}{\epsilon}
\newcommand{\w}{\omega}

\newcommand{\G}{\Gamma}

\newcommand{\ee}{\mathbf{e}}

\newcommand{\bdelta}{\boldsymbol{\delta}}
\newcommand{\bkappa}{\boldsymbol{\kappa}}

\newcommand{\zbar}{\overline{z}}

\newcommand{\deltabar}{\overline{\delta}}

\newcommand{\xibar}{\overline{\xi}}
\newcommand{\wbar}{\overline{\w}}
\newcommand{\stilde}{\tilde{s}}
\newcommand{\Omegatilde}{\widetilde{\Omega}}
\newcommand{\Htilde}{\widetilde{H}}
\newcommand{\Dtilde}{\widetilde{D}}
\newcommand{\Ihat}{\hat{I}}

\renewcommand{\hbar}{\overline{h}}

\newcommand{\Xbar}{{\overline{X}}}

\newcommand{\del}{\partial}
\newcommand{\delbar}{\overline{\partial}}
\newcommand{\tbar}{\overline{t}}

\newcommand{\Ch}{\mathrm{Ch}}

\newcommand{\MHS}{\mathsf{MHS}}

\newcommand{\un}{\mathrm{un}}

\newcommand{\blank}{\phantom{x}}
\newcommand{\round}[1]{{(#1)}}

\newcommand{\bdot}{\bullet}

\newcommand{\bs}{\backslash}

\newcommand\im{\operatorname{im}}               
\newcommand\id{\operatorname{id}}
\newcommand\coker{\operatorname{coker}}

\newcommand\Hom{\operatorname{Hom}}
\newcommand\End{\operatorname{End}}
\newcommand\Ext{\operatorname{Ext}}
\newcommand\Aut{\operatorname{Aut}}

\newcommand\Gr{\operatorname{Gr}}
\newcommand\Res{\operatorname{Res}}

\renewcommand\Im{\operatorname{Im}}
\renewcommand\Re{\operatorname{Re}}

\newtheorem{theorem}{Theorem}[section]
\newtheorem{lemma}[theorem]{Lemma}
\newtheorem{proposition}[theorem]{Proposition}
\newtheorem{corollary}[theorem]{Corollary}
\newtheorem{bigtheorem}{Theorem}

\theoremstyle{definition}

\theoremstyle{remark}
\newtheorem{remark}[theorem]{Remark}

\begin{document}

\title{Periods of Real Biextensions}

\author{Richard Hain}
\address{Department of Mathematics\\ Duke University\\ Durham, NC 27708-0320}
\email{hain@math.duke.edu}

\thanks{ORCID iD: {\sf 0000-0002-7009-6971}}

\date{\today}



\maketitle

\begin{abstract}
A real biextension is a mixed Hodge structure that is an extension of $\R$ by a mixed Hodge structure with weights $-1$ and $-2$. It is {\em split} if it is a direct sum of its graded quotients. A unipotent real biextension over an algebraic manifold $X$ is a variation of mixed Hodge structure over $X$, each of whose fibers is a real biextension and whose weight graded quotients are constant. In this paper we show that if a unipotent real biextension has non-abelian monodromy, then its ``general fiber'' is not split. This result is a tool for investigating the boundary behaviour of normal functions. 
\end{abstract}

\section{Introduction}

Suppose that $A$ is a noetherian subring of $\R$. An $A$-biextension $V$ is an $A$ mixed Hodge structure (MHS) with three non-trivial weight graded quotients of weights $0,-1$ and $-2$. There is no loss of generality in assuming, as we do, that the $0$th weight graded quotient of $V$ is the Hodge structure $A(0)$ of type $(0,0)$.

Real biextensions are, by definition, $\R$-biextensions. In this paper we study the period mappings of families of real biextensions with unipotent monodromy over a smooth variety.\footnote{All varieties in this paper are complex and quasi-projective.} Families of real biextensions whose weight $-2$ graded quotient is $\R(1)$ and their period mappings appear naturally in the study of the archimedean component of height pairings of algebraic cycles. (See \cite{hain:biext,bloch-dejong,dejong:jumps}, for example.) The more general biextensions studied in this paper arise naturally when investigating the boundary behaviour of normal functions, as explained below.

A {\em unipotent $A$-biextension} over a smooth variety $X$ is an admissible variation $\bV$ of $A$-MHS over $X$ with unipotent monodromy whose fibers are biextensions. The unipotence of the monodromy is equivalent to the weight graded quotients of $\bV$ being constant variations of Hodge structure, \cite[p.~84]{hain-zucker}. When $A=\R$ we say that $\bV$ is a {\em unipotent real biextension}.

\begin{bigtheorem}
\label{thm:main}
Suppose that $X$ is a quasi-projective manifold and that $\bV$ is a unipotent real biextension over $X$. If either
\begin{enumerate}

\item the monodromy representation of $\bV$ has non-abelian image, or

\item the monodromy group of $\bV$ is abelian and $\Gr^W_{-2} H_1(X)$ acts non-trivially on the fiber over $x\in X$,

\end{enumerate}
then there is a nowhere dense real analytic subvariety $\Sigma$ of $X$ such that the fiber $V_x$ over $x\in X$ splits if and only if $x\in\Sigma$.
\end{bigtheorem}

The most fundamental examples of unipotent biextensions over $X$ are the truncated path torsors discussed in Sections~\ref{sec:unipt_paths}, \ref{sec:de_rham} and \ref{sec:hodge_path}. They are, in a sense, universal and play a central role in the proof.

Real variations of MHS behave very differently from variations of $\Q$-MHS. If a variation of $\Q$-MHS over a smooth variety has non-trivial monodromy, then the MHS on the general fiber does not split as a MHS. This is not the case for real variations of MHS as every extension of $\R$ by a variation of Hodge structure of weight $-1$ splits regardless of the monodromy representation. For this reason, it is not clear that the general fiber of a unipotent real biextension with non-abelian monodromy does not split, as the theorem asserts.

The theorem should be a useful tool for understanding the boundary behaviour of normal functions. An admissible normal function $\nu$ over a smooth variety determines a ``residual variation" of MHS along the smooth locus of each of its boundary divisors, \cite[Rem.~13.3]{hain:ceresa}. Each of these is an extension of $\Q$ by an admissible variation of MHS $\bL$ with strictly negative weights, so that $\bL/W_{-3}$ is a biextension. If, for example, the restriction of $\bL$ to a subvariety $Y$ of a boundary divisor is a unipotent variation with non-abelian monodromy, as occurs in the case of the normal function $\nu$ of the Ceresa cycle (see \cite[\S15]{hain:ceresa}), the theorem implies that the normal derivative of $\nu$ at a ``general point'' of the boundary divisor is non-zero. This is because, under some mild hypotheses, the normal derivative vanishes at a point $y\in Y$ if and only if the associated {\em real} biextension $\bL/W_{-3}$ splits there. The theorem is used in \cite{hain:ceresa} to show that, when $g>3$, the normal derivative of the Ceresa normal function is non-zero at a general point of the boundary divisor $\Delta_0$ of $\cM_g$, the moduli space of curves.

\subsection{Overview}

A real biextension\footnote{As remarked above there is no loss of generality in assuming that $\Gr^W_0 V = A$. One can always replace a mixed Hodge structure $E$ with weight graded quotients $\Gr^W_0 E = A_0$, $\Gr^W_{-1} E = A_1$ and $\Gr^W_{-2} E = A_2$ by $\Hom_A(A_0,E)$ and then restricting to those elements that project to a multiple of $\id_{A_0}$ in $\Gr^W_0\Hom_A(A_0,E)$ to obtain a biextension $V$ with weight graded quotients $A$, $\Hom(A_0,A_1)$ and $\Hom(A_0,A_2)$. The original MHS $E$ can be recovered as a sub MHS of $A_0 \otimes V$.} $V$ with weight graded quotients
\begin{equation}
\label{eqn:GrV}
B = \Gr^W_{-1} V \text{ and } C = \Gr^W_{-2} V
\end{equation}
is determined by its period, which is an element of $iC_\R^{-1,-1}$. (See Section~\ref{sec:biexts}.) It vanishes if and only if $V$ splits as an $\R$-MHS. So a unipotent real biextension $\bV$ over $X$ with weight graded quotients (\ref{eqn:GrV}) determines a period mapping
$$
\Psi_\bV : X \to iC_\R^{-1,-1}.
$$
The fiber of $\bV$ over $q \in X$ splits if and only if $\Psi_\bV(q)=0$. One of the main tasks of the paper is to find a reasonably explicit formula for this map.

To obtain this formula, we first study the most fundamental example of a unipotent biextension associated to $X$. It is the unipotent variation of MHS over $X\times X$ whose fiber over $(p,q)$ is
\begin{equation}
\label{eqn:path_torsor}
\Z\pi_1(X;p,q)/W_{-3},
\end{equation}
where $\pi_1(X;p,q)$ is the set of homotopy classes of paths in $X$ from $p$ to $q$. This variation and its generalizations were first studied in \cite{hain-zucker} where they are called the {\em canonical variations} of MHS over $X^2$. The construction of the MHS on (\ref{eqn:path_torsor}) is recalled in Section~\ref{sec:hodge_path} after reviewing the relevant de~Rham theorem in Section~\ref{sec:de_rham}.

The formula for the period of the biextension $\R\pi_1(X;p,q)/W_{-3}$ is derived in Section~\ref{sec:period}. For clarity, we first derive it when $X$ is projective where the formula is simpler because the Hodge structure $H^1(X)$ is pure of weight 1. 

The next step in the proof is to establish a ``non-degeneracy'' result (Theorem~\ref{thm:non-degeneracy}) for the period mapping of the real biextension quotient (\ref{eqn:path_torsor}) of the canonical variation. This is achieved in Section~\ref{sec:non-degen} using the period computations in Section~\ref{sec:period}. The curve case is proved first; the general case follows from it via the Lefschetz hyperplane theorem.

To derive the formula for the period map of a general unipotent real biextension and to complete the proof of the theorem, we need the ``well-known'' fact that if $\bV$ is an admissible unipotent variation of MHS over $X$, then, for $m$ sufficiently large, the map
$$
A\pi_1(X;p,q)/W_{-m}\otimes V_p \to V_q
$$
induced by parallel transport is a morphism of MHS. Since I could not find a proof in the literature, I have provided one in Section~\ref{sec:vmhs}. As explained in Section~\ref{sec:proof}, this implies that the period mapping of a unipotent real biextension over $X$ can be computed from the period mapping of the (second real) canonical variation of $X$ and its monodromy homomorphism. The main theorem follows from this and the non-degeneracy result.

\begin{remark}
One should also be able to obtain the period computation in Section~\ref{sec:period} using Goncharov's Hodge correlators, \cite{goncharov}. That said, our derivation is direct and relatively elementary.
\end{remark}

\subsection{A note to the reader}

Those unfamiliar with iterated integrals, Chen's $\pi_1$ de~Rham theorem and the associated Hodge theory should consult the expository article \cite{hain:bowdoin}. Complete proofs of many of the technical statements needed in this paper can be found in \cite{chen}, \cite{hain:dht} and \cite{hain-zucker}. A nice exposition of how to compute periods of extensions of Hodge structures can be found in Carlson's paper \cite{carlson}.

\subsection{Conventions}
All algebraic varieties will be complex algebraic varieties. Every quasi-projective variety will be assumed to be irreducible. Consequently, every quasi-projective manifold is connected when endowed with the complex topology.

\bigskip

\noindent{\em Acknowledgments:} I am indebted to the referees for their careful reading of the manuscript and their helpful comments and corrections. I would also like to think Robert Bryant for pointing out the relevance of the Cartan--K\"ahler theorem.

\section{Moduli of biextensions}
\label{sec:biexts}

Suppose that $B$ and $C$ are $A$-Hodge structures of weights $-1$ and $-2$, respectively. Assume that, as $A$-modules, $\Ext^1_A(A,B)$, $\Ext^1_A(A,C)$ and $\Ext^1_A(B,C)$ vanish. This will hold whenever $B$ and $C$ are free $A$-modules. Throughout we will regard $A$ as the Hodge structure $A(0)$ of type $(0,0)$.

Define $\sB_A(B,C)$ to be the set of {\em framed} biextensions with weight graded quotients $A$, $B$ and $C$. This is the set of $A$-MHS $V$ satisfying
$$
V = W_0 V \text{ and } W_{-3} V = 0
$$
endowed with isomorphisms
$$
A \overset{\simeq}{\longrightarrow} \Gr^W_0 V,\ B \overset{\simeq}{\longrightarrow} \Gr^W_{-1} V,\ C \overset{\simeq}{\longrightarrow} \Gr^W_{-2} V.
$$
It has a canonical base point, namely the split biextension
$$
V_0 := A \oplus B \oplus C.
$$
Denote the fraction field of $A$ by $\kk$. Set $G = W_{-1} \Aut V_0$. This is the affine $\kk$-group where, for each $\kk$ algebra $R$,
$$
G(R) = \{\phi \in \Aut V_0\otimes_A R : \Gr^W_\bdot \phi = \id_{\Gr^W_\bdot V}\}.
$$
Its Lie algebra is
$$
\g := W_{-1} \End V_0\otimes_A \kk.
$$
It has a natural MHS. Define $F^0 G$ to be the subgroup of $G\times_\kk \C$ with Lie algebra $F^0 \g$.

There is a natural bijection
$$
\sB_A(B,C) \overset{\simeq}{\longrightarrow} G(A) \bs G(\C) / F^0 G(\C).
$$
(See \cite[Prop.~5.10]{hain-zucker}.) This induces a natural topology on $\sB_A(B,C)$ and a complex structure on it when $A=\Z$. The projection
\begin{equation}
\label{eqn:torsor}
\sB_A(B,C) \to \Ext^1_\MHS(A,B)\times \Ext^1_\MHS(B,C)
\end{equation}
defined by $V \mapsto (V/W_{-2},W_{-1}V)$ is an $\Ext^1_\MHS(A,C)$ torsor.

The following is well known, but included as it is important for the rest of this paper.

\begin{proposition}
\label{prop:real_MHS}
\phantom{xx}
\begin{enumerate}

\item Each real Hodge structure $E$ of weight $m$ decomposes
$$
E = \bigoplus_{\substack{p+q = m\\ p\ge q}} E(p,q)
$$
where $E(p,q)_\C = E^{p,q} \oplus E^{q,p}$ when $p\neq q$, and $E(p,p)_\C = E^{p,p}$.

\item If $E$ and $K$ are real Hodge structures of weights $m$ and $k$, then
$$
\Ext_\MHS^1(E,K) \cong
\begin{cases}

0 &  k \ge m-1, \\

i\Hom(E,K)_\R^{-1,-1} & k = m-2.

\end{cases}
$$
\end{enumerate}
\end{proposition}

\begin{proof}
In all cases complex conjugation acts on $E(p,q)$ and
$$
E(p,q) = E(p,q)_\R \oplus i E(p,q)_\R
$$
where $E(p,q)_\R$ denotes the elements fixed by complex conjugation. The proof of (1) is left as an exercise. To prove the second, use the fact that
$$
\Ext_\MHS^1(E,K) \cong \Ext^1(\R,\Hom(E,K))
$$
and, for a real Hodge structure $M$,
$$
\Ext_\MHS^1(\R,M) = W_0 M_\C/(W_0 M_\R + F^0 W_0 M).
$$
\end{proof}

One consequence is that when $A=\R$, the base of the torsor (\ref{eqn:torsor}) is a point.

\begin{corollary}
\label{cor:real-biexts}
If $B$ and $C$ are real Hodge structures, then there are natural isomorphisms
$$
\sB_\R(B,C) \cong \Ext^1_\MHS(\R,C) \cong i C_\R^{-1,-1}.
$$
\end{corollary}

The {\em period} of an element of $\sB_\R(B,C)$ is the corresponding element of $iC_\R^{-1,-1}$. It determines the biextension and  vanishes if and only if it splits.

\subsection{How to compute the period of a real biextension}
\label{sec:recipe}

Here we give a recipe for computing the period of a real biextension. Suppose that $V$ is an $\R$ biextension with weight graded quotients $\R$, $B$ and $C$, as above.

Consider the dual biextension $V^\vee := \Hom_\R(V,\R)$. It has weight graded quotients $\R$, $B^\vee$ and $C^\vee$ of weights $0,1$ and 2, respectively. The first step is to compute a real splitting $s : C^\vee \to V^\vee/\R$ of the extension
$$
0 \to B^\vee \to V^\vee/\R \to C^\vee \to 0
$$
in the category of real MHS. This lifts to a real section
$$
\stilde_\R : C_\R^\vee \to V_\R^\vee \text{ of the quotient map } V_\R^\vee \to C_\R^\vee
$$
and a Hodge filtration preserving section
$$
\stilde_F : C_\C^\vee \to V_\C^\vee \text{ of the quotient map } V_\C^\vee \to C_\C^\vee.
$$
The first is unique up to an element of $\Hom_\R(C_\R^\vee,\R) \cong C_\R$. The second is unique up to an element of $F^0 \Hom_\C(C^\vee,\C) \cong F^0 C$. Their difference $\stilde_F - \stilde_\R$ is an element of $C_\C \cong \Hom(C_\C^\vee,\C)$ whose image in
$$
C_\C/(C_\R + F^0 C) \cong iC_\R^{-1,-1}
$$
is well defined. It is the period of the biextension.

\section{Unipotent quotients of path torsors}
\label{sec:unipt_paths}

Suppose that $X$ is a connected, smooth manifold and that $p,q\in X$. Denote the set of homotopy classes of paths in $X$ from $p$ to $q$ by $\pi_1(X;p,q)$. It is a torsor (i.e., a principal homogeneous space) under both $\pi_1(X,p)$ and $\pi_1(X,q)$, which act by pre- and post-composition, respectively. Suppose that $R$ is a commutative ring. Denote the free $R$-module generated by $\pi_1(X;p,q)$ by $R\pi_1(X;p,q)$. This is a free left $R\pi_1(X,p)$ module of rank 1. Taking each element of $\pi_1(X;p,q)$ to 1 defines an augmentation
$$
\epsilon_{p,q} : R\pi_1(X;p,q) \to R.
$$
Denote its kernel by $I_{p,q}$. This extends to a filtration
$$
R\pi_1(X;p,q) = I^0_{p,q} \supseteq I_{p,q} \supseteq I^2_{p,q} \supseteq \cdots
$$
where
$$
I^m_{p,q} := I_p^m R\pi_1(X;p,q) = I_p^{m-1} I_{p,q}\quad \text{when } m\ge 1
$$
and  $I_p$ denotes the augmentation ideal of $R\pi_1(X,p)$. Since $R\pi_1(X;p,q)$ is a free $R\pi_1(X,p)$ module of rank 1 generated by any path $\gamma$ from $p$ to $q$, there is, for each $m\ge 0$, a natural isomorphism
$$
\Gr_I^m R\pi_1(X;p,q) \cong I_p^m/I_p^{m+1}.
$$

\begin{remark}
For each $m > 0$ there is a local system over $X\times X$ whose fiber over $(p,q)$ is $R\pi_1(X;p,q)/I_{p,q}^{m+1}$. The monodromy representation
$$
\pi_1(X,p) \times \pi_1(X,q) \to \Aut \big( R\pi_1(X;p,q)/I_{p,q}^{m+1} \big)
$$
is given by $(\mu_p,\mu_q) : \gamma \to \mu_p \gamma \mu_q^{-1}$. It is filtered by the sub local systems whose fiber over $(p,q)$ is $I_{p,q}^j/I_{p,q}^{m+1}$. Each graded quotient of the local system is trivial so that the global monodromy representation is unipotent.
\end{remark}

If $R$ is a PID and $H_1(X;R)$ is torsion free, then
$$
\Gr_I^m R\pi_1(X;p,q) \cong
\begin{cases}
R & m = 0; \cr
H_1(X;R) & m = 1; \cr
\coker\{H_2(X;R) \overset{\delta_X}{\to} H_1(X;R)^{\otimes 2}\} & m = 2,
\end{cases}
$$
where $\delta_X$ is induced by the diagonal map $X \to X\times X$. (See \cite{stallings} or \cite[Lem.~6.2]{hain:bowdoin}.) It is the dual of the cup product map when $H_1(X;R)$ is finitely generated. More generally, each $\Gr_I^m R\pi_1(X;p,q)$ is a quotient of $H_1(X;R)^{\otimes m}$. So if $H_1(X;R)$ is finitely generated, then $R\pi_1(X;p,q)/I_{p,q}^{m+1}$ is finitely generated for all $m \ge 0$.

\section{A de~Rham theorem}
\label{sec:de_rham}

Assume now that $H_1(X;\R)$ is finite dimensional so that $\R\pi_1(X;p,q)/I_{p,q}^{m+1}$ is finite dimensional. General results of K.-T.~Chen \cite{chen} imply that its dual space
$$
\Hom_\R(\R\pi_1(X;p,q)/I_{p,q}^{m+1},\R)
$$
can be computed using differential forms. We now recall briefly how this works in the case $m=2$. Basic references for this material are \cite{chen} and \cite{hain:bowdoin}.

Denote the space of piecewise smooth paths in $X$ from $p$ to $q$ by $P_{p,q}X$. Let $\kk$ denote $\R$ or $\C$. A $\kk$-valued iterated integral is an expression of the form
\begin{equation}
\label{eqn:it-int}
F = c + \int \xi + \sum c_{jk} \int \varphi_j \varphi_k
\end{equation}
where $\xi$ and each $\varphi_j$ are smooth 1-forms on $X$ and $c,c_{jk}\in \kk$. The iterated integral $F$ is regarded as the function $F : P_{p,q} X \to \kk$ that takes the value
$$
F(\gamma) = c + \int_\gamma \xi + \sum c_{jk} \int_\gamma \varphi_j \varphi_k
$$
on $\gamma$, where
$$
\int_\gamma \varphi_j \varphi_k := \int_{0\le t_1 \le t_2 \le 1} f_j(t_1) f_k(t_2)  dt_1 dt_2
$$
and $\gamma^\ast \varphi_j = f_j(t) dt$.

Such an iterated integral is said to be {\em relatively closed} (or a {\em homotopy functional}) if its value on a path depends only on its homotopy class relative to its endpoints.

\begin{lemma}
\label{lem:closed}
If each $\varphi_j$ is closed, the iterated integral (\ref{eqn:it-int}) is relatively closed if and only if
\begin{equation}
\label{eqn:integrability}
d\xi + \sum c_{jk}\,\varphi_j\wedge \varphi_k = 0.
\end{equation}
\end{lemma}

This is well known and can be proved by working on the universal covering of $X$. For another proof, see \cite[Prop.~3.1]{hain:bowdoin}.

Denote the space of relatively closed $\kk$-valued iterated integrals of the form (\ref{eqn:it-int}) by $\Ch_2(X;p,q)_\kk$. It has a natural filtration\footnote{Although it may appear that this filtration may have a natural splitting by tensor degree, it does not. For example if $h$ is a smooth function on $X$ and $\varphi$ a 1-form, then
$$
\int df = (f(q)-f(p))\epsilon_{p,q} \text{ and } \int dh\, \varphi = \int (h\varphi) - h(p)\int \varphi.
$$
This will be important in the proof of the main result.}
\begin{equation}
\label{eqn:filt}
\Ch_2(X;p,q)_\kk \supseteq \Ch_1(X;p,q)_\kk \supseteq \Ch_0(X;p,q)_\kk,
\end{equation}
where $\Ch_1(X;p,q)_\kk$ consists of the relatively closed iterated integrals
$$
c + \int \xi
$$
of length $\le 1$ (i.e., $d\xi = 0$) and let $\Ch_0(X)_\kk = \kk\epsilon_{p.q}$. The iterated integrals (\ref{eqn:it-int}) satisfying (\ref{eqn:integrability}) span $\Ch_2(X;p,q)_\kk$.

\begin{proposition}
\label{prop:chen-dr}
Integration defines an isomorphism
$$
\Ch_2(X;p,q)_\kk \to \Hom_\kk(\kk \pi_1(X;p,q)/I_{p,q}^3,\kk).
$$
The filtration (\ref{eqn:filt}) is dual to the filtration
$$
\kk\pi_1(X;p,q)/I_{p,q}^3 \supseteq I_{p,q}(X;p,q)/I_{p,q}^3 \supseteq I_{p,q}^2(X;p,q)/I_{p,q}^3 \supseteq 0
$$
of $\kk\pi_1(X;p,q)/I_{p,q}^3$.
\end{proposition}

\begin{proof}
The case $p=q$ is a special case of Chen's $\pi_1$ de~Rham theorem. (See Cor.~1 of \cite[Thm.~2.6.1]{chen} or \cite[Thm.~4.4]{hain:bowdoin}.) The case when $p\neq q$ follows from this via \cite[\S3]{hain-zucker}. Indeed, the groups $\Ch_2(X;p,q)_\kk$ form a flat vector bundle over $X\times X$. Integration induces a map from the sheaf of flat sections of this vector bundle to the locally constant sheaf over $X\times X$ whose fiber over $(p,q)$ is $\Hom_\kk(\kk \pi_1(X;p,q)/I_{p,q}^3,\kk)$. Since the restriction of this map to the diagonal (or even just one point on the diagonal) is an isomorphism, it is an isomorphism on all fibers.
\end{proof}

There is a well defined mapping $\kappa : \Ch_2(X;p,q)_\kk \to H^1(X;\kk)^{\otimes 2}$. With the notation of (\ref{eqn:it-int}), it is defined by
\begin{equation}
\label{eqn:kappa}
\kappa : c + \int \xi + \sum c_{jk} \int \varphi_j \varphi_k \mapsto \sum c_{jk} [\varphi_j]\otimes[\varphi_k].
\end{equation}
The following is a direct consequence of Lemma~\ref{lem:closed} and the definition of $\Ch_1(X;p,q)$.

\begin{lemma}
\label{lem:exact}
The sequence
$$
0 \to \Ch_1(X;p,q)_\kk \to \Ch_2(X;p,q)_\kk \overset{\kappa}{\to} \ker\{H^1(X;\kk)^{\otimes 2} \overset{\text{cup}}{\longrightarrow} H^2(X;\kk)\} \to 0
$$
is exact.
\end{lemma}

We can decompose $H^1(X;\kk)^{\otimes 2}$ into its symmetric and skew-symmetric parts:
$$
H^1(X;\kk)^{\otimes 2} = S^2 H^1(X;\kk) \oplus \Lambda^2 H^1(X;\kk).
$$
This restricts to the decomposition
\begin{multline}
\label{eqn:decomp_K}
\ker\{H^1(X;\kk)^{\otimes 2} \overset{\text{cup}}{\longrightarrow} H^2(X;\kk)\} 
\cr
= S^2 H^1(X;\kk) \oplus \ker\{\Lambda^2 H^1(X;\kk) \overset{\text{cup}}{\longrightarrow} H^2(X;\kk)\}.
\end{multline}
The element $[\varphi]\otimes [\psi] + [\psi]\otimes [\varphi]$ of $S^2 H^1(X;\kk)$ lifts to the ``decomposable'' iterated integral
$$
\int \varphi \int \psi = \int \varphi\psi + \int \psi\varphi.
$$

\section{The Hodge biextension associated to a path torsor}
\label{sec:hodge_path}

\begin{theorem}[\cite{hain-zucker}]
If $X$ is a quasi-projective manifold and $p\in X$, then, for all $m > 0$, the local system over $X$ with fiber $\Z\pi_1(X;p,q)/I_{p,q}^m$ over $q\in X$ underlies an admissible unipotent variation of $\Z$-MHS over $X$. For all $j\ge 0$, the weight filtration on each fiber satisfies
$$
W_{-j} \Q\pi_1(X;p,q)/I_{p,q}^m \supseteq I_{p,q}^j/I_{p,q}^m
$$
with equality if and only if $H^1(X)$ is pure of weight 1.
\end{theorem}

This section is a review of the construction of the canonical $\Z$-MHS on $\Z\pi_1(X;p,q)/W_{-3}$. An exposition of the construction of the MHS on $\Z\pi_1(X;p,q)/I_{p,q}^m$ when $p=q$ can be found in \cite{hain:bowdoin} and complete details in \cite{hain:dht,hain-zucker}. We do this by constructing an $\R$-MHS on its dual, $W_2\Ch_2(X;p,q)$. We first recall the construction in the projective case as it is simpler.

\subsection{The projective case}

In this section $X$ is a projective manifold. The weight filtration
$$
0 \subseteq W_0 \Ch_2(X;p,q) \subseteq W_1 \Ch_2(X;p,q) \subseteq W_2 \Ch_2(X;p,q) = \Ch_2(X;p,q)
$$
is defined by $W_j \Ch_2(X;p,q) = \Ch_j(X;p,q)$. It is dual to the weight filtration of $\kk\pi_1(X;p,q)/I_{p,q}^3$, which is defined by
$$
W_{-j} \kk\pi_1(X;p,q)/I_{p,q}^3 = I_{p,q}^j/I_{p,q}^3,\quad 0 \le j \le 3.
$$
Denote the complex of smooth $\C$-valued forms on $X$ by $E^\bdot(X)$. The Hodge filtration
$$
\Ch_2(X;p,q)_\C = F^0 \Ch_2(X;p,q) \supseteq F^1 \Ch_2(X;p,q) \supseteq F^2 \Ch_2(X;p,q) \supseteq 0
$$
is defined by insisting that $F^s\Ch_2(X;p,q)$ is spanned by the relatively closed iterated integrals
$$
c \epsilon_{p,q} + \int \xi + \sum c_{jk} \int \varphi_j \varphi_k,
$$
where each $\varphi_j$ is closed, $\xi \in F^s E^1(X)$ and
$$
\sum c_{jk}\, \varphi_j \otimes \varphi_k \in F^s\big(E^1(X)\otimes E^1(X)\big).
$$
We can and will assume that each $\varphi_j$ is holomorphic or anti holomorphic. Note that the map $\kappa$ defined in (\ref{eqn:kappa}) is a morphism of MHS.

\subsection{The general case}
\label{sec:gen_hodge}

Now assume that $X$ is a quasi-projective manifold. Let $\Xbar$ be a smooth projective completion of $X$ such that $X = \Xbar-D$, where $D$ is a divisor with normal crossings. Denote the $C^\infty$ log complex of $(\Xbar,D)$ by $E^\bdot(\Xbar\log D)$.

Suppose that $\kk$ is a subfield of $\R$. In the category of $\kk$-Hodge structures we have isomorphisms
$$
\Gr^W_{-1} \kk \pi_1(X;p,q)/W_{-3} \cong \Gr^W_{-1} H_1(X;\kk) \cong  H_1(\Xbar;\kk)
$$
and an exact sequence \cite[Lem.~6.1]{hain:bowdoin}
\begin{equation}
\label{eqn:extension}
0 \to H_1(\Xbar;\kk)^{\otimes 2}/\im\deltabar_X \to \Gr^W_{-2} \kk\pi_1(X;p,q)/W_{-3} \to \Gr^W_{-2}H_1(X;\kk) \to 0
\end{equation}
where $\deltabar_X$ is the morphism of MHS defined by the commutative diagram:
$$
\begin{tikzcd}
H_2(X;\kk) \ar[r,"\delta_X"] \ar[d] \ar[dr,"\deltabar_X"] & H_1(X;\kk)^{\otimes 2} \ar[d] \\
H_2(\Xbar;\kk) \ar[r,"\delta_\Xbar"] & H_1(\Xbar;\kk)^{\otimes 2}
\end{tikzcd}
$$
Note that, for all topological spaces $Z$, the image of $\delta_Z$ is contained in the subspace $\Lambda^2 H_1(Z;\kk)$ of $H_1(Z;\kk)^{\otimes 2}$. In particular, this holds when $Z$ is $X$ and $\Xbar$.

The dual of $\kk\pi_1(X;p,q)/W_{-3}$ is the subspace $W_2\Ch_2(X;p,q)$ of $\Ch_2(X;p,q)$ spanned by the relatively closed iterated integrals of the form
$$
c + \int \xi + \sum c_{jk} \int \varphi_j \varphi_k,
$$
where $\xi \in E^1(\Xbar\log D)$ and each $\varphi_j$ or its complex conjugate is a holomorphic 1-form on $\Xbar$. Its weight filtration is defined by $W_0 \Ch_2(X;p,q) = \C \epsilon_{p,q}$,
$$
W_1 \Ch_2(X;p,q) = \big\{c + \int \varphi : \varphi \text{ or its complex conjugate } \in H^0(\Xbar,\Omega_\Xbar^1)\big\}.
$$
The $s$th term $F^sW_2\Ch_2(X;p,q)$ of the Hodge filtration of $W_2\Ch_2(X;p,q)$ is the span of the relatively closed iterated integrals
$$
c + \int \xi + \sum c_{jk} \int \varphi_j \varphi_k,
$$
where each $\varphi_j$ or its complex conjugate is in $H^0(\Xbar,\Omega_\Xbar^1)$, $\xi \in F^s E^1(\Xbar\log D)$ and
$$
\sum c_{jk}\, \varphi_j \otimes \varphi_k \in F^s\big(E^1(\Xbar\log D)\otimes E^1(\Xbar\log D)\big).
$$
The map $\kappa$ and the sequence in Lemma~\ref{lem:exact} are morphisms of MHS.

Set
\begin{equation}
\label{eqn:def_K}
K_\kk = \ker\{H^1(\Xbar;\kk)^{\otimes 2} \overset{j^\ast\circ\text{cup}}{\longrightarrow} H^2(X;\kk)\},
\end{equation}
where $j : X \hookrightarrow \Xbar$ is the inclusion. It is a Hodge structure of weight 2 and is isomorphic to the dual of $H_1(\Xbar;\kk)^{\otimes 2}/\im\deltabar_X$, where $\deltabar_X$ is the map defined above. Taking $\Gr^W_2$ of (\ref{eqn:decomp_K}), we see that
\begin{equation}
\label{eqn:decomposition}
K_\kk = S^2 H^1(\Xbar;\kk) \oplus \big(K_\kk\cap \Lambda^2 H^1(\Xbar;\kk)\big).
\end{equation}
The dual of (\ref{eqn:extension}) is the exact sequence
\begin{equation}
\label{eqn:dual}
0 \to \Gr^W_2 H^1(X;\kk) \to \Gr^W_2\Ch_2(X;p,q)_\kk \to K_\kk \to 0
\end{equation}
of Hodge structures. The last map is $\Gr^W_2\kappa$. The sequence can also be obtained by taking $\Gr^W_2$ of the exact sequence in Lemma~\ref{lem:exact}.

\begin{remark}
\label{rem:real_split}
When $p=q$, the real mixed Hodge structure on $\Z\pi_1(X;p,q)/W_{-3}$ splits. This is because
$$
\Z\pi_1(X,p)/W_{-3} = \Z \oplus I_p/W_{-3}
$$
in the category of $\Z$-MHS and because the real MHS on $I_p/W_{-3}$ splits by Proposition~\ref{prop:real_MHS}.
\end{remark}

\section{The period of the MHS on $\R\pi_1(X;p,q)/W_{-3}$}
\label{sec:period}

Throughout this section, $\kk$ will denote $\R$ or $\C$. We'll compute the period using the recipe in Section~\ref{sec:recipe}. The dual of the biextension is $W_2\Ch_2(X;p,q)$. As in the previous section, we will consider the projective case first as it is simpler and contains all of the main ideas.

\subsection{The projective case}

For all $m\ge 0$, integration defines an isomorphism
$$
\Ch_m(X;p,q)_\kk/\Ch_0(X;p,q)_\kk \overset{\simeq}{\longrightarrow} \Hom_\kk(I_{p,q}/I_{p,q}^{m+1}, \kk).
$$

In the projective case, $K_\kk$ is the kernel of the cup product $H^1(X;\kk)^{\otimes 2} \to H^2(X;\kk)$. Our first task is to compute the unique splitting
$$
s : K \to \Ch_2(X;p,q)/\Ch_0(X;p,q)
$$
of the surjection $\Ch_2(X;p,q)/\Ch_0(X;p,q) \to K$ in the category of real MHS as well as real and Hodge filtration preserving lifts
$$
\stilde_\R : K \to \Ch_2(X;p,q)_\R \text{ and } \stilde_F : K_\C \to \Ch_2(X;p,q)_\C
$$
of $s$. Their difference $\stilde_F - \stilde_\R$ then determines the period.

Recall the notation of Proposition~\ref{prop:real_MHS}. In particular, $K=K(1,1) \oplus K(2,0)$.

\subsubsection{Computation of the sections over $K(2,0)$}
\label{sec:K20}

We first construct the sections over $K^{2,0}$. Fix a basis $\{\w_j\}$ of $H^0(\Omega^1_X)$. Since $H^0(\Omega_X^2) \cong H^{2,0}(X)$, a holomorphic 2-form on $X$ is exact if and only if it vanishes. This means that
$$
\sum a_{jk} \omega_j \otimes \omega_k \in H^0(\Omega_X)^{\otimes 2},\quad a_{jk} \in \C
$$
represents an element of $K^{2,0}$ if and only if
$$
\sum a_{jk} \omega_j \wedge \omega_k = 0 \in H^0(\Omega_X^2).
$$
By Lemma~\ref{lem:closed}, the iterated integral
$$
\sum a_{jk} \int \omega_j \omega_k
$$
is relatively closed and determines an element of $\Ch_2(X;p,q)$. This defines a section
$$
\stilde^{2,0} : K^{2,0} \to F^2\Ch_2(X;p,q)
$$
of $F^2\Ch_2(X;p,q) \to K^{2,0}$. Its complex conjugate is a section
$$
\stilde^{0,2} : K^{0,2} \to F^0\Ch_2(X;p,q).
$$
Their sum defines a morphism of real mixed Hodge structures $\stilde : K(2,0) \to \Ch_2(X;p,q)$.

\subsubsection{Computation of the sections over $K(1,1)$}
\label{sec:compn}

It suffices to define the sections on $K_\R^{1,1}$. In view of (\ref{eqn:decomposition}), a representative of each element of $K_\R^{1,1}$ can be decomposed as a sum
$$
\sum h_{jk}\, (\w_j \otimes \wbar_k - \wbar_k \otimes \w_j) + \sum a_{jk} (\w_j \otimes \wbar_k + \wbar_k \otimes \w_j)
$$
where $\hbar_{jk} = - h_{kj}$ and $\overline{a}_{jk} = a_{kj}$. The first term is an element of $K_\R^{1,1} \cap \Lambda^2 H^1(X;\R)$ and the second of $[S^2 H^1(X)]_\R^{1,1}$. We compute the sections on the two terms separately.

Define $\stilde$ on elements of the second type by defining
\begin{align}
\label{eqn:decomp}
\stilde\big(\sum a_{jk} (\w_j \otimes \wbar_k + \wbar_k \otimes \w_j)\big)
&= \sum a_{jk} \int (\w_j \wbar_k + \wbar_k \w_j) \cr
&= \sum a_{jk} \int \w_j \int \wbar_k.
\end{align}
This lies in $F^1\Ch_2(X;p,q) \cap \Ch_2(X;p,q)_\R$. So on $[S^2 H^1(X)]_\R^{1,1}$, we can define the value of $\stilde_\R$ and $\stilde_F$ to be the value of $\stilde$.

The interesting part of the period is supported on $K_\R^{1,1} \cap \Lambda^2 H^1(X;\R)$. Suppose that
\begin{equation}
\label{eqn:omega}
\Omegatilde = \sum h_{jk}\, (\w_j \otimes \wbar_k - \wbar_k \otimes \w_j) \in K_\R^{1,1} \cap \Lambda^2 H^1(X;\R).
\end{equation}
Its image in $E^2(X)$ is $2\Omega$, where $\Omega = \sum h_{jk}\, \w_j \wedge \wbar_k$. By the $\del\delbar$-lemma, there is a smooth function $f : X \to \C$ such that
$$
\Omega = i\del\delbar f.
$$
Since $\Omega$ is real analytic, the Cartan--K\"ahler Theorem \cite[Thm.~2.2]{eds} implies that $f$ is also real analytic. Since $\Omega$ and $i\del\delbar$ are real, we can (and will) take $f$ to be real. It is unique up to a real constant. Set
$$
\xi = i\del f \text{ and } \varphi = \xi + \xibar = i(\del f - \delbar f).
$$
Then $\xi$ is a $(1,0)$-form and $\varphi$ a real 1-form that satisfy
$$
\Omega + \delbar \xi = 0,\ 2\Omega + d\varphi = 0 \text{ and } 2\xi-\varphi = i df.
$$
Define
$$
\stilde_F(\Omegatilde) := 2\int \xi + \sum h_{jk} \int (\w_j \wbar_k - \wbar_k \w_j) \in F^1 \Ch_2(X;p,q)
$$
and
$$
\stilde_\R(\Omegatilde) := \int \varphi + \sum h_{jk} \int (\w_j \wbar_k - \wbar_k \w_j) \in \Ch_2(X;p,q)_\R.
$$
These satisfy
$$
\stilde_F(\Omegatilde) = \stilde_\R(\Omegatilde) + i(f(q)-f(p))\epsilon_{p,q} \in \Ch_2(X;p,q)_\R
$$
so that $\stilde_F \equiv \stilde_\R \bmod \Ch_0(X;p,q)$. Together with the section (\ref{eqn:decomp}), they define a morphism of real mixed Hodge structures
$$
s : K(1,1) \to \Ch_2(X;p,q)/\Ch_0(X;p,q)
$$
as well as two lifts
$$
\stilde_F : K(1,1) \to F^1\Ch_2(X;p,q) \text{ and } \stilde_\R : K(1,1)_\R \to \Ch_2(X;p,q)_\R
$$
of it. Their difference vanishes on $[S^2 H^1(X)]_\R^{1,1}$.

\subsubsection{Computation of the period}

Recall that the image of $\delta_X$ lies in $\Lambda^2 H_1(X;\kk)$. The decomposition
$$
H_1(X;\kk)^{\otimes 2}/\im \delta_X \cong S^2 H_1(X;\kk) \oplus (\Lambda^2 H_1(X;\kk)/\im \delta_X)
$$
is dual to (\ref{eqn:decomposition}). Since the restriction of the sections $\stilde_F$ and $\stilde_\R$ to $S^2 H^1(X)$ are equal, we have:

\begin{proposition}
\label{prop:image}
If $X$ is a projective manifold, the image of the period mapping
$$
\Psi : X\times X \to i[H_1(X)^{\otimes 2}/\im \delta_X]_\R^{-1,-1}
$$
is contained in $i[\Lambda^2 H_1(X)/\im\delta_X]_\R^{-1,-1}$.
\end{proposition}

Choose a basis $\{\Omegatilde_j\}$ of $[\Lambda^2 H^1(X)]_\R^{1,1} \cap K_\R$ and real-valued real functions $f_j : X \to \R$ such that $i\del\delbar f_j = \Omega_j$. Each $f_j$ is real analytic and unique up to a constant. Denote the dual basis of $[\Lambda^2 H_1(X)/\im \delta_X]_\R^{-1,-1}$ by $\{\bkappa_j\}$. Define
$$
\Psi : X\times X \to i[\Lambda^2 H_1(X)/\im \delta_X]_\R^{-1,-1}
$$
by $\Psi(p,q) = i\sum_j \big(f_j(q)-f_j(p)\big) \bkappa_j$. It vanishes on the diagonal.

\begin{proposition}
The function $\Psi$ is the period mapping of the real biextension over $X^2$ whose fiber over $(p,q)$ is $\R\pi_1(X;p,q)/I_{p,q}^3$. It is real analytic.
\end{proposition}

\begin{proof}
Apply the recipe given in Section~\ref{sec:recipe} using the computations in Sections~\ref{sec:K20} and \ref{sec:compn}. Analicity follows from the fact that each $f_j$ is real analytic.
\end{proof}

\subsection{The general case}

In this section we give the period computation in the general case. Some details similar to those in the projective case above are left to the reader. We will use the setup and notation of Section~\ref{sec:gen_hodge}. In particular, $X=\Xbar-D$, where $\Xbar$ is a projective manifold and $D$ is a divisor with normal crossings.

Denote the normalization of $D$ by $\Dtilde$. The Gysin sequence
$$
0 \to H^1(\Xbar) \to H^1(X) \overset{\Res}{\to} H_0(\Dtilde)(-1) \to H^2(\Xbar) \to H^2(X) \to \cdots
$$
is an exact sequence of MHS. This implies that $W_1H^1(X) = H^1(\Xbar)$ (as noted above) and that
$$
\Gr^W_2 H^1(X) = \ker\{H_0(\Dtilde)(-1) \to H^2(\Xbar)\}.
$$

\subsubsection{The Hodge structure $M_X$}

For $p,q\in X$, define $M_X$ via the pullback diagram
$$
\begin{tikzcd}
M_X \ar[r]\ar[d] & \Gr^W_{-2} \Q\pi_1(X;p,q) \ar[d,twoheadrightarrow] \\
(\Lambda^2 H_1(\Xbar;\Q))/\im\delta_\Xbar \ar[r,hookrightarrow] & \Gr^W_{-2} \Q\pi_1(\Xbar;p,q).
\end{tikzcd}
$$
It is a Hodge structure of weight $-2$ and does not depend on $(p,q)\in X^2$. It is an extension
\begin{equation}
\label{eqn:ext_M}
0 \to \Lambda^2 H_2(\Xbar;\Q)/\im\deltabar_X \to M_X \to \Gr^W_{-2}H_1(X;\Q) \to 0.
\end{equation}

The next result is the generalization of Proposition~\ref{prop:image} to the case when $X$ is not necessarily projective.

\begin{proposition}
\label{prop:image_general}
The image of the period map
$$
\Psi : X^2 \to i[\Gr^W_{-2} \R\pi_1(X;p,q)]^{-1,-1}
$$
is contained in $i[M_X]^{-1,-1}_\R$.
\end{proposition}

Henceforth, we will regard the period mapping $\Psi$ as a map $\Psi : X^2 \to i[M_X]^{-1,-1}_\R$.

\begin{proof}
Set $M_\Xbar = \Lambda^2 H_1(\Xbar;\Q)/\im \delta_\Xbar$. By Proposition~\ref{prop:image}, the image of the period mapping
$$
\Psi_\Xbar : \Xbar^2 \to i[\Gr^W_{-2} \R\pi_1(\Xbar;p,q)]^{-1,-1}
$$
of $\Xbar$ lies in $i[M_\Xbar]_\R^{-1,-1}$. Since the diagram
$$
\begin{tikzcd}
X^2 \ar[rr,"\Psi"] \ar[d] && i{[\Gr^W_{-2} \R\pi_1(X,p)]^{-1,-1}} \ar[d]\\
\Xbar^2 \ar[r] \ar[rr,bend right,"\Psi_\Xbar"] & i{[M_\Xbar]_\R^{-1,-1}} \ar[r,hookrightarrow] & i{[\Gr^W_{-2} \R\pi_1(\Xbar;p,q)]^{-1-1}}
\end{tikzcd}
$$
commutes, the image of $\Psi$ is contained in $i[M_X]^{-1,-1}_\R$, as claimed.

\end{proof}

\subsubsection{Computation of the period map}

Suppose that $Z$ is a subvariety of $\Xbar$ of dimension $r$. Denote by $\bdelta_{[Z]}$ the de~Rham current on $\Xbar$ that takes a smooth $2r$-form $\eta$ on $\Xbar$ to its integral $\int_{Z'}\eta$ over the smooth locus of $Z$. This definition extends to algebraic cycles with coefficients in $\C$ by linearity.

We use the notation (\ref{eqn:omega}) for elements of $K\cap \Lambda^2 H^1(\Xbar)$ and the images in $E^2(\Xbar)$ of their harmonic representatives. The dual of (\ref{eqn:ext_M}) is the exact sequence
$$
0 \to \Gr^W_2 H^1(X) \to M_X^\vee \to K\cap \Lambda^2 H^1(\Xbar;\R) \to 0.
$$
Let $\{\Omegatilde_a\}$ be a basis of $K_\R^{1,1}\cap \Lambda^2 H^1(\Xbar)$ and $\{\bkappa_a\}$ be the dual basis of $[\Lambda^2 H_1(\Xbar)/\im \deltabar_X]_\R^{-1,-1}$. The exactness of the sequence above implies that, for each $a$, we can choose an element $E_a$ of $H_0(\Dtilde;\R)$ such that $\Omega_a$ represents the Poincar\'e dual of $E_a$. Since $\Omega_a$ and $E_a$ are real, there is a 0-current $f_a$ on $\Xbar$ of logarithmic type\footnote{See \cite[\S1.3]{gillet-soule} for background on currents.} that is a smooth real-valued function away from the support of $E_a$ and satisfies
$$
i\del\delbar f_a = \Omega_a - \bdelta_{[E_a]}.
$$
It is real analytic and well defined up to the addition of a real constant. The restriction of $f_a$ to $X$ does not depend on the choice of compactification $\Xbar$ as can be seen by adapting the proof \cite[p.~37]{deligne:h2} of \cite[Thm.~3.2.5(ii)]{deligne:h2}.

Note that $\del f_a \in F^1E^1(\Xbar\log D)$. Write
$$
\Omegatilde_a = \sum h_{jk}^\round{a}  (\w_j \otimes \wbar_k - \wbar_k \otimes \w_j) \ \in [K \cap \Lambda^2 H^1(\Xbar)]_\R^{1,1}.
$$
It lifts to the relatively closed iterated integral
$$
2\int i\del f_a + \sum h_{jk}^\round{a} \int (\w_j \wbar_k - \wbar_k \w_j) \in F^1W_2\Ch_2(X;p,q)
$$
and
$$
\int i(\del f_a - \delbar f_a) + \sum h_{jk}^\round{a} \int (\w_j \wbar_k - \wbar_k \w_j) \in W_2\Ch_2(X;p,q)_\R.
$$
Since $2\del f_a - (\del f_a - \delbar f_a) = df_a$, these differ by
$$
\int_p^q idf_a = i\big(f_a(q)-f_a(p)\big)\e_{p,q} \in \Ch_0(X;p,q),
$$
so that their images in $\Ch_2(X;p,q)/\Ch_0(X;p,q)$ are equal. Denote their common image in $\Gr^W_2\Ch_2(X;p,q)$ by $I_{\Omegatilde_a}$.

Since the real Hodge structure on $H^1(X)$ splits, there is a linearly independent subset $\{\zeta_k\}$ of $H^0(\Omega_\Xbar^1(\log D))$ that projects to a basis of $\Gr^W_2 H^1(X)$, where each $\zeta_k$ has real periods. The set $\{\zeta_k,I_{\Omegatilde_a}\}$ is a basis of $[M_X^\vee]_\R^{1,1}$. The set $\{\bkappa_a\}$ can be built up to the dual basis $\{\ee_k,\bkappa_a\}$ of $[M_X]_\R^{-1,-1}$.

\begin{proposition}
\label{prop:period}
The period mapping
$$
\Psi : X \times X \to i[M_X]_\R^{-1,-1}
$$
of the biextension over $X^2$ with fiber $\R\pi_1(X;p,q)/I^3$ over $(p,q)\in X^2$ is
$$
\Psi(p,q) = i\Big(\sum_k \Im \int_p^q \zeta_k \,\ee_k + \sum_a \big(f_a(q)-f_a(p)\big) \, \bkappa_a\Big).
$$
It is real analytic.
\end{proposition}

\subsection{Affine curves}
\label{sec:affine}

Here we elaborate on the period computation in the case of affine curves as it plays a crucial role in the proof of Theorem~\ref{thm:main}.

Suppose that $X = \Xbar-D$, where $\Xbar$ is a smooth projective curve and $D=\{x_0,\dots,x_m\}$ is a non-empty finite subset. Fix $p\in X$.

Since $H^2(X)=0$, the Gysin sequence implies that
$$
\Gr^W_1 H^1(X) \cong H^1(\Xbar) \text{ and } \Gr^W_2 H^1(X) \cong \Htilde_0(D).
$$
For each $\Omegatilde \in [\Lambda^2 H^1(\Xbar)]_\R^{1,1}$ there is a real-valued function $f$ on $\Xbar-\{x_0\}$ whose restriction to a neighbourhood of $x_0$ in $\Xbar$ satisfies
$$
f = h - \frac{1}{2\pi} \bigg(\int_\Xbar \Omega\bigg) \log |z|^2,
$$
where $h$ is smooth and $z$ is a local holomorphic coordinate about $x_0$ and,  when regarded as a current on $\Xbar$,
\begin{equation}
\label{eqn:def_f}
i\del\delbar f = \Omega - \bigg(\int_\Xbar \Omega\bigg) \bdelta_{[x_0]}.
\end{equation}
The function $f$ is real harmonic and unique up to a real constant. For $p\in X$, we normalize $f$ by setting $f_\Omega = f - f(p)$.

For each $k = 1,\dots,m$, there is a {\em unique} differential of the third kind $\zeta_k$ with real periods that is holomorphic except at $x_0$ and $x_k$ where it has simple poles with residues $-(2\pi i)^{-1}$ and $(2\pi i)^{-1}$, respectively. Their images in $\Gr^W_2 H^1(X)$ form a basis of it. Define $h_k : X \to \R$ by
$$
h_k(q) = \Im \int_p^q \zeta_k.
$$
It is a well defined real analytic function as the periods of $\zeta_k$ are real.

Choose a basis $\{\Omegatilde_a\}$ of $[\Lambda^2 H^1(\Xbar)]_\R^{1,1}$ and let $I_{\Omegatilde_a}$ be the lift of $\Omegatilde_a$ to $\Gr^W_2 \Ch_2(X;p,q)_\R$ defined in the previous section. Then $\{\zeta_k,I_{\Omegatilde_a}\}$ is a basis of $[M_X^\vee]_\R^{1,1}$. Let $\{\ee_k,\bkappa_a\}$ be the dual basis of $[M_X]_\R^{-1,-1}$.

With this notation, the period mapping $\Psi_p : X \to i[M_X]_\R^{-1,-1}$ of the biextension with fiber $\R\pi_1(X;p,q)/W_{-3}$ over $q\in X$ is the real analytic function $\Psi_p$ defined by
$$
\Psi_p(q) = i\Big(\sum_k h_k(q)\,\ee_k + \sum_a f_{\Omega_a}(q) \, \bkappa_a\Big).
$$

\section{Unipotent completion}
\label{sec:unipt}

This section is a terse review of unipotent completion and the associated Hodge theory. Details can be found in \cite[\S2.5]{hain:dht}. Additional background can be found in \cite[App.~A]{quillen}. See \cite{hain:bowdoin} for a leisurely exposition.

Suppose that $\G$ is a discrete group and that $\kk$ is a field of characteristic zero. The group algebra $\kk\G$ is a cocommutative Hopf algebra. The diagonal $\Delta : \kk\G \to \kk\G \otimes \kk\G$ is defined by $\Delta : \gamma \to \gamma\otimes\gamma$ for all $\gamma \in \G$. The powers of its augmentation ideal $I$ define a topology on $\kk\G$. Its $I$-adic completion is the complete algebra
$$
\kk\G^\wedge := \varprojlim_n \kk\G/I^n.
$$
The diagonal $\Delta$ is continuous in the $I$-adic topology and induces a diagonal
$$
\Delta : \kk\G^\wedge \to \kk\G^\wedge \widehat{\otimes} \kk\G^\wedge.
$$
This gives $\kk\G^\wedge$ the structure of a cocommutative {\em complete} Hopf algebra.

The map $\G \to I/I^2$ that takes $\gamma\in \G$ to $(\gamma-1) + I^2$ is a group homomorphism. It induces an isomorphism $H_1(\G,\kk) \to I/I^2$.

Now suppose that $H_1(\G,\Q)$ is finite dimensional. The {\em unipotent completion} (also called the {\em Malcev completion}) $\G^\un$ of $\G$ is a prounipotent group defined over $\Q$. Its set of $\kk$ rational points is the set of group-like elements of $\kk\G^\wedge$:
$$
\G^\un(\kk) = \{u\in \kk\G^\wedge : \Delta u = u\otimes u\}.
$$
Its Lie algebra $\p$ is the set of primitive elements of $\Q\G^\wedge$:
$$
\p = \{x \in \Q\G^\wedge: \Delta x = 1\otimes x + x \otimes 1\}.
$$
This is a pronilpotent Lie algebra that is contained in the completed augmentation ideal $\hat{I}$. Its lower central series is
$$
\p \supseteq \p^2 \supseteq \p^3 \supseteq \cdots
$$
where $\p^n := \p \cap \hat{I}^m$. Each $\p/\p^m$ is a finite dimensional nilpotent Lie algebra. There are natural isomorphisms
$$
H_1(\p) := \p/\p^2 \cong H_1(\G,\Q) \text{ and } \p^2/\p^3 \cong [\Lambda^2 H_1(\G,\Q)]/\im\delta
$$
such that the diagram
$$
\begin{tikzcd}
\Lambda^2 \p  \ar[d] \ar[r,"{[\blank,\blank]}"] & \p^2 \ar[d] \\
\Lambda^2 H_1(\G,\Q) \ar[r] & {[\Lambda^2 H_1(\G,\Q)]}/\im\delta
\end{tikzcd}
$$
commutes, where $\delta : H_2(\G,\Q) \to \Lambda^2 H_1(\G,\Q)$ is the dual of the cup product.

\begin{theorem}[{\cite[Thm.~6.3.1]{hain:dht}}]
Suppose that $X$ is a complex algebraic variety and that $p\in X$. The Lie algebra $\p(X,p)$ of the unipotent completion of $\pi_1(X,p)$ has a natural pro-MHS whose bracket is a morphism. The canonical isomorphisms
$$
H_1(\p(X,p)) \cong H_1(X;\Q) \text{ and } \p(X,p)^2/\p(X,p)^3 \cong [\Lambda^2 H_1(X)]/\im \delta_X
$$
are isomorphisms of MHS.
\end{theorem}

When $X$ is a quasi-projective manifold, $W_{-m} \p(X,p) \subseteq \p(X,p)^m$ with equality if and only if $H^1(X)$ is a Hodge structure of weight 1.

\begin{lemma}
\label{lem:Gr_p}
If $X$ is a quasi-projective manifold, then $\Gr^W_{-2} \p(X,p)= 0$ if and only if $\Gr^W_{-2}H_1(X) = 0$ and $\deltabar_X : H_2(X) \to \Lambda^2 H_1(\Xbar)$ is surjective.
\end{lemma}

\begin{proof}
This result follows from the exactness of the sequence
$$
0 \to \Lambda^2 H_1(\Xbar;\Q)/\im \deltabar_X \to \Gr^W_{-2}\p(X,x) \to \Gr^W_{-2} H_1(X;\Q) \to 0
$$
which is obtained by restriction from (\ref{eqn:extension}) with $p=q$.
\end{proof}

The only quasi-projective curves with $\Gr^W_{-2}\p(X,p)=0$ are $\P^1, \C$ and elliptic curves.

\begin{proposition}
If $X$ is a quasi-projective manifold, then $M_X = \Gr^W_{-2}\p(X,p)$.
\end{proposition}

\begin{proof}
Denote the completion of the augmentation ideal of $\Q\pi_1(X,p)$ by $\Ihat$. The projection
$$
\Ihat \to \Ihat/\Ihat^2 \cong I/I^2 \cong H_1(X;\Q)
$$
is a morphism of Hodge structures, It induces a surjection $\Gr^W_{-2}\Q\pi_1(X,p)^\wedge \to \Gr^W_{-2} H_1(X;\Q)$. Both $M_X$ and $\Gr^W_{-2}\p(X,p)$ surject onto $\Gr^W_{-2} H_1(X;\Q)$. The kernel of both projections is $\Lambda^2 H_1(X;\Q)/\im\deltabar_X$. It follows that $M_X = \Gr^W_{-2}\p(X,p)$.
\end{proof}

\section{Unipotent variations of MHS}
\label{sec:vmhs}

In this section, we recall some basic facts about unipotent variations of MHS that we will need in the next section. A basic reference for this section is \cite{hain-zucker}. Throughout $X$ will be a quasi-projective manifold.

Suppose $A$ is a noetherian subring of $\R$ and $\kk$ its fraction field. Suppose that $\bV$ is an admissible unipotent variation of $A$-MHS over $X$.\footnote{These are called ``good'' unipotent variations of MHS in \cite{hain-zucker}. When $X$ is projective, the only condition that admissibility imposes is Griffiths transversality.} The unipotence of $\bV$ is equivalent to the condition that each $\Gr^W_m \bV$ is a constant variation of Hodge structure.

For each $(p,q)\in X^2$, there is a monodromy action
$$
A\pi_1(X;p,q) \to \Hom(V_p,V_q)
$$
defined by parallel transport. Since the monodromy is unipotent, there is a positive integer $m$ such that the kernel contains $I_{p,q}^{m+1}$. We thus have an $A$-module homomorphism
\begin{equation}
\label{eqn:monod}
\rho_{p,q} : A\pi_1(X;p,q)/I_{p,q}^{m+1} \to \Hom_A(V_p,V_q).
\end{equation}
The following is well known, at least to experts. It is a consequence of the main theorem of \cite{hain-zucker}.

\begin{lemma}
\label{lem:morphism}
For all $(p,q)\in X^2$, the monodromy homomorphism (\ref{eqn:monod}) is a morphism of $A$-MHS.
\end{lemma}

\begin{proof}
Fix $p\in X$. The homomorphism (\ref{eqn:monod}) is a morphism from the local system over $X$ with fiber $A\pi_1(X;p,q)/I_{p,q}^{m+1}$ over $q$ to the local system with fiber $\Hom_A(V_p,V_q)$ over $q$. Both underlie admissible unipotent variations of $A$-MHS over $X$. When $q=p$, the homomorphism is the monodromy representation of $\bV$. It is a morphism of MHS by the main result of \cite{hain-zucker}, which also implies that the map of local systems is a morphism of variations of MHS.
\end{proof}

When $p=q$, the monodromy action (\ref{eqn:monod}) restricts to the infinitesimal monodromy action
\begin{equation}
\label{eqn:inf_monod}
\p(X,p)/\p(X,p)^{m+1} \to \End V_p.
\end{equation}

\section{Non-degeneracy of unipotent path torsors}
\label{sec:non-degen}

The key step in the proof of Theorem~\ref{thm:main} is to prove the following non-degeneracy result, which is prove in this section.

\begin{theorem}
\label{thm:non-degeneracy}
If $X$ is a quasi-projective manifold, then for all $p\in X$ the period mapping $\Psi_p : X \to i[\Gr^W_{-2}\p(X,p)]_\R^{-1,-1}$ of the variation with fiber $\R\pi_1(X;p,q)/W_{-3}$ over $q \in X$ is non-degenerate in the sense that its image is not contained in any real hyperplane in $i[\Gr^W_{-2}\p(X,p)]_\R^{-1,-1}$.
\end{theorem}

There is nothing to prove when $\Gr^W_{-2}\p(X,p) = 0$. So we will suppose that $\Gr^W_{-2}\p(X,p)\neq 0$. We first consider the case where $X$ is a curve. The general case will follow from the curve case and the Lefschetz hyperplane theorem.

\subsection{The curve case}

\begin{lemma}
\label{lem:dependent}
Suppose that $m,n\ge 0$ and that $f_1,\dots,f_n, h_1,\dots,h_m$ are holomorphic functions on the unit disk. If
\begin{equation}
\label{eqn:equal}
\sum_{j=1}^n |f_j|^2 = \sum_{k=1}^m |h_k|^2,
\end{equation}
then $\{f_1,\dots,f_n,h_1,\dots,h_m\}$ is linearly dependent.
\end{lemma}

\begin{proof}
Denote the holomorphic coordinate in the disk by $t$. Suppose that
$$
f_j(t) = \sum_{r=0}^\infty a_r^\round{j} t^r \text{ and } h_k(t) = \sum_{r=0}^\infty b_r^\round{k} t^r.
$$
The power series expansion of $|f_j(t)|^2$ and $|h_k(t)|^2$ in terms of $t$ and $\tbar$ converge absolutely and are unique. Since
\begin{align*}
\sum_{j=1}^n |f_j(t)|^2 - \sum_{k=1}^m |h_k(t)|^2
&= \sum_{j=1}^n f_j(t)\overline{f_j(t)} - \sum_{k=1}^m h_k(t)\overline{h_k(t)} \cr
&= \sum_{r=0}^\infty \Big(\sum_{j=1}^n a_r^\round{j} \overline{f_j(t)} - \sum_{k=1}^m b_r^\round{k} \overline{h_k(t)}\Big)t^r,
\end{align*}
the uniqueness of the power series expansions and the vanishing of (\ref{eqn:equal}) implies that, for each $r$, we have
$$
\sum_{j=1}^n a_r^\round{j} \overline{f_j(t)} - \sum_{k=1}^m b_r^\round{k} \overline{h_k(t)} = 0.
$$
If all of the $a_r^\round{j}$ and $b_r^\round{k}$ are zero, all of the functions are identically zero and are thus linearly dependent. If any $a_r^\round{j}$ or $b_r^\round{k}$ is non-zero, then, after taking complex conjugates, we obtain a non-trivial linear relation between $f_1,\dots,f_n,h_1,\dots,h_m$.
\end{proof}

Suppose that $\Xbar$ is a compact Riemann surface of genus $g > 0$ and that $X=\Xbar-D$, where $D$ is finite. Fix a basis $\w_1,\dots,\w_g$ of $H^0(\Omega_\Xbar^1)$. Every element of $[\Lambda^2 H^1(\Xbar)]_\R^{1,1}$ is represented by an element of $E^1(\Xbar)_\C^{\otimes 2}$ of the form
$$
\Omegatilde = \sum h_{jk}\, (\w_j \otimes \wbar_k - \wbar_k \otimes \w_j)
$$
where the matrix $(h_{jk})$ is skew hermitian. The image of $\Omegatilde$ in $E^2(X)$ is $2 \Omega$, where
$$
\Omega = \sum h_{jk}\, \w_j \wedge \wbar_k.
$$

\begin{lemma}
\label{lem:restriction}
The differential 2-form $\Omega$ vanishes on $X$ if and only if $\Omegatilde$ vanishes.
\end{lemma}

\begin{proof}
It is clear that $\Omega$ vanishes when $\Omegatilde$ vanishes. So assume that $\Omega$ vanishes. Since $i(h_{jk})$ is hermitian, we can, after a change of basis, assume that it is diagonal with diagonal entries in $\{0,\pm 1\}$. In this case
$$
\Omega = i \sum_{j=1}^a \w_j \wedge \wbar_j - i \sum_{j=a+1}^{a+b} \w_j \wedge \wbar_j
$$
and
$$
\Omegatilde = \sum_{j=1}^a (\w_j \otimes \wbar_k - \wbar_k \otimes \w_j) - \sum_{j=a+1}^{a+b} (\w_j \otimes \wbar_k - \wbar_k \otimes \w_j),
$$
where $0 \le a+b\le g$. Let $U$ be a coordinate disk in $X$ with holomorphic coordinate $z$. The restriction of $\Omega$ to $U$ is
$$
\sum_{j=1}^a |f_j(z)|^2 i\, dz\wedge d\zbar - \sum_{j=a+1}^{a+b} |f_j(z)|^2 i\, dz\wedge d\zbar,
$$
where $\w_j|_U = f_j(z)dz$, which vanishes as $\Omega=0$. Since $\w_1,\dots,\w_g$ are linearly independent, so are $f_1,\dots,f_g$. Lemma~\ref{lem:dependent} implies that $a=b=0$, so that $\Omegatilde=0$.
\end{proof}

\begin{proposition}
\label{prop:curve-case}
If $X$ is a quasi-projective curve with $\Gr^W_{-2}\p(X,p) \neq 0$, then the period mapping $\Psi_p$ is non-degenerate. The vanishing locus of $\Psi_p$ is a nowhere dense real analytic subvariety of $X$.
\end{proposition}

\begin{proof}
Denote the smooth projective completion of $X$ by $\Xbar$. The computations in Section~\ref{sec:affine} and the exactness of the sequence (\ref{eqn:extension}) and its dual (\ref{eqn:dual}) imply that hyperplanes in $[\Gr^W_{-2} \p(X,p)]_\R^{-1,-1}$ are defined by non-zero pairs $(\Omegatilde,\zeta)$, where $\Omegatilde \in [K\cap \Lambda^2 H^1(\Xbar)]_\R^{1,1}$ and $\zeta$ is a differential of the 3rd kind on $\Xbar$ with real periods.

The period computation in Section~\ref{sec:affine} implies that the image of $\Psi_p$ lies in the zero locus of the function corresponding to $(\Omegatilde,\zeta)$ if and only if the function $F : X \to \R$ defined by
$$
F: q \mapsto f(q) + \Im\int_p^q \zeta
$$
vanishes on $X$, where $f$ satisfies (\ref{eqn:def_f}). To prove that $\Psi_p$ is non-degenerate, we need to prove that $F=0$ implies that $\Omegatilde$ and $\zeta$ both vanish. Suppose that $F=0$. Write $\zeta = \Re\zeta + idh$ where
$$
h(q) = \Im\int_p^q \zeta.
$$
Since $h$ is harmonic on $X$ we have
$$
i\del\delbar F = i\del\delbar h + i\del\delbar f = i\del\delbar f = \Omega \in E^2(X).
$$
Since $F=0$, this implies that $\Omega = 0$ and, by Lemma~\ref{lem:restriction}, that $\Omegatilde = 0$. Thus $F=h$. Since $\zeta$ is holomorphic, the vanishing of $h$ implies that $\zeta = 0$. The second assertion follows as $\Psi_p$ is a non-zero real analytic function on $X$.
\end{proof}

\subsection{Proof of Theorem~\ref{thm:non-degeneracy}}

It suffices to show that $\Psi_p$ is non-degenerate for all $p\in X$. We can and will assume that $X$ is a Zariski open subset of a smooth projective subvariety $\Xbar$ of $\P^N$ of dimension $>1$.

The intersection $T$ of $X$ with a generic linear subspace of $\P^N$ of codimension $\dim X-1$ that contains $p$ is a smooth curve in $X$. By the Lefschetz hyperplane theorem \cite[p.~150]{smt}, $\pi_1(T,p) \to \pi_1(X,p)$ is surjective. This implies that the induced map $\p(T,p)\to \p(X,p)$ is also surjective. Since it is a morphism of MHS,
$$
\Gr^W_{-2} \p(T,p) \to \Gr^W_{-2}\p(X,p)
$$
is surjective. And since $\kk\pi_1(T;p,q)/W_{-3} \to \kk\pi_1(X;p,q)/W_{-3}$ is a morphism of MHS for all $q\in T$, the diagram
$$
\begin{tikzcd}
T \ar[r,"\Psi_p^T"] \ar[d,hookrightarrow] & i[\Gr^W_{-2} \p(T,p)]_\R^{-1,-1} \ar[d,twoheadrightarrow] \\
X \ar[r,"\Psi_p"] & i[\Gr^W_{-2} \p(X,p)]_\R^{-1,-1}
\end{tikzcd}
$$
commutes, where the top map is the period map of the variation over $T$ with fiber $\R\pi_1(T;p,q)/W_{-3}$ over $q$. If $\Psi_p$ were degenerate, then so would $\Psi_p^T$. But Proposition~\ref{prop:curve-case} implies that $\Psi_p^T$ is non-degenerate, which implies the non-degeneracy of $\Psi_p$.

\section{Proof of Theorem~\ref{thm:main}}
\label{sec:proof}

As in the statement, $X$ is a quasi-projective manifold. Suppose that $\bV$ is a real biextension over $X$ with weight graded quotients the constant variations $B$ and $C$ as in (\ref{eqn:GrV}). To prove Theorem~\ref{thm:main}, it suffices to show that the period mapping $\Psi_\bV : X \to iC_\R^{-1,-1}$ is not identically zero.

Fix $p\in X$. Lemma~\ref{lem:morphism} implies that the monodromy homomorphism (\ref{eqn:inf_monod}) is a morphism of MHS. It induces a homomorphism of Hodge structures
$$
\Phi_p : \Gr^W_{-2} \p(X,p) \to C.
$$
If either condition in Theorem~\ref{thm:main} holds, then $\Phi_p$ is non-zero. Theorem~\ref{thm:non-degeneracy} implies that $\Phi_p$ is non-degenerate, so the real analytic map
$$
\begin{tikzcd}
X \ar[r,"\Psi_p"] & i[\Gr^W_{-2} \p(X,p)]_\R^{-1,-1} \ar[r,"\Phi_p"] & iC_\R^{-1,-1}
\end{tikzcd}
$$
cannot vanish identically. Theorem~\ref{thm:main} now follows from:

\begin{lemma}
The period map $\Psi_\bV : X \to i C_\R^{-1,-1}$ of $\bV$ is the real analytic map
$$
\Psi_\bV = \Psi_\bV(p) + \Phi_p\circ \Psi_p.
$$
\end{lemma}

\begin{proof}
We first suppose that the MHS on $V_p$ splits. Denote the generator of $\Gr^W_0 V_p$ that corresponds to $1\in \R$ by $\ee_p$. For all $q\in X$, the homomorphism
$$
\R\pi_1(X;p,q)/W_{-3} \to V_q
$$
defined by $u \mapsto u\ee_p$ is a morphism of MHS. It follows that the period of $V_q$ is $\Phi_p\circ \Psi_p(q)$. Since the real MHS on $\R\pi_1(X,p)/W_{-3}$ splits (Remark~\ref{rem:real_split}), $\Psi_p(p) = 0$. So the formula holds in this case.

To treat the general case, recall from Corollary~\ref{cor:real-biexts} that real biextensions with weight graded quotients $\R$, $B$ and $C$ form an abelian group. So we can modify $\bV$ by twisting it by the biextension with period $-\Psi_\bV(p) \in iC_\R^{-1,-1} \cong \Ext^1_\MHS(\R,C)$. The new biextension $\bV_p$ splits over $p$. Its period mapping is $\Psi_\bV - \Psi_\bV(p)$. The first part implies that
$$
\Psi_\bV - \Psi_\bV(p) = \Phi_p\circ \Psi_p
$$
from which the result follows.
\end{proof}


\begin{thebibliography}{99}

\bibitem{bloch-dejong}
S.~Bloch, R.~de Jong, E.~Sert\"oz:
{\em Heights on curves and limits of Hodge structures}, J.\ Lond.\ Math.\ Soc.\ (2) 108 (2023), 340--361.

\bibitem{eds}
R.~Bryant, S.-S.~Chern, R.~Gardner, H.~Goldschmidt, P.~Griffiths:
{\em Exterior differential systems}. Mathematical Sciences Research Institute Publications, 18. Springer-Verlag, 1991. 

\bibitem{carlson}
J.~Carlson:
{\em Extensions of mixed Hodge structures}. Journ\'ees de G\'eometrie Alg\'ebrique d'Angers, Juillet 1979, pp.~107--127, Sijthoff \& Noordhoff, 1980.

\bibitem{gillet-soule}
H.~Gillet, C.~Soul\'e:
{\em Arithmetic intersection theory}, Inst.\ Hautes \'Etudes Sci.\ Publ.\ Math.\ No.~72 (1990), 93--174 (1991).

\bibitem{chen}
K.-T.~Chen:
{\em Iterated path integrals}, Bull.\ Amer.\ Math.\ Soc.\ 83 (1977), 831--879.

\bibitem{deligne:h2}
P.~Deligne:
{\em Th\'eorie de Hodge, II}. Inst.\ Hautes \'Etudes Sci.\ Publ.\ Math.\ No.~40 (1971), 5--57.

\bibitem{smt}
M.~Goresky, R.~MacPherson:
{\em Stratified Morse theory}. Ergebnisse der Mathematik und ihrer Grenzgebiete (3) 14. Springer-Verlag, Berlin, 1988.

\bibitem{goncharov}
A.~Goncharov:
{\em Hodge correlators}, J. Reine Angew.\ Math.\ 748 (2019), 1--138.

\bibitem{hain:bowdoin}
R.~Hain:
{\em The geometry of the mixed Hodge structure on the fundamental group}, Algebraic geometry, Bowdoin, 1985,  247--282, Proc.\ Sympos.\ Pure Math., 46, Part 2, Amer.\ Math.\ Soc., 1987. 

\bibitem{hain:dht}
R.~Hain:
{\em The de Rham homotopy theory of complex algebraic varieties, I}, K-Theory 1 (1987), 271--324.

\bibitem{hain:biext}
R.~Hain:
{\em Biextensions and heights associated to curves of odd genus}, Duke Math.\ J.\ 61 (1990), 859--898.

\bibitem{hain:ceresa}
R.~Hain:
{\em The rank of the normal functions of the Ceresa and Gross--Schoen cycles}, {\sf [arXiv:2408.07809]}

\bibitem{hain-zucker}
R.~Hain, S.~Zucker:
{\em Unipotent variations of mixed Hodge structure}, Invent.\ Math.\ 88 (1987), 83--124.

\bibitem{dejong:jumps}
R.~de~Jong, F.~Shokrieh:
{\em Jumps in the height of the Ceresa cycle}, J.\ Differential Geom.\ 126 (2024), 169--214.

\bibitem{quillen}
D.~Quillen:
{\em Rational homotopy theory}, Ann.\ of Math.\  90 (1969), 205--295.

\bibitem{stallings}
J.~Stallings:
{\em Homology and central series of groups}, J.\ Algebra 2 (1965), 170--181.

\end{thebibliography}
\end{document}